\documentclass[11pt]{article}

\usepackage[T1]{fontenc}
\usepackage[utf8]{inputenc}
\usepackage[margin=2.9cm]{geometry}
\usepackage{microtype}
\usepackage{amsmath,amssymb,amsthm,mathtools,comment}
\usepackage{xcolor}
\usepackage[backref=page]{hyperref}
\usepackage{enumitem}
\hypersetup{
  colorlinks=true,
  linkcolor={blue!85!black},
  citecolor={red!85!black},
  urlcolor={blue!55!black}
}
\allowdisplaybreaks
\theoremstyle{plain}
\newtheorem{theorem}{Theorem}[section]
\newtheorem{lemma}[theorem]{Lemma}
\newtheorem{proposition}[theorem]{Proposition}

\newtheorem{question}[theorem]{Question}
\newtheorem{conjecture}[theorem]{Conjecture}
\newtheorem{observation}[theorem]{Observation}
\newtheorem{problem}[theorem]{Problem}
\theoremstyle{definition}
\newtheorem{remark}[theorem]{Remark}
\providecommand{\keywords}[1]{\textbf{Keywords:} #1}
\providecommand{\subjclass}[1]{\textbf{Mathematics Subject Classifications:} #1}

\newcommand{\ceil}[1]{\left\lceil #1\right\rceil}
\newcommand{\floor}[1]{\left\lfloor #1\right\rfloor}

\newcommand{\Dmin}{\Delta^{\min}}
\newcommand{\Dmax}{\Delta^{\max}}
\newcommand{\delmin}{\delta^{\min}}
\newcommand{\delmax}{\delta^{\max}}
\newcommand{\dmin}{d^{\min}}
\newcommand{\dmax}{d^{\max}}

\newcounter{hsstatement}

\newcounter{ohsstatement}

\title{The Hajnal--Szemer\'edi theorem in digraphs revisited}
\author{Yufei Chang\thanks{
School of Mathematics, Shandong University, Jinan 250100, China.
E-mail: \texttt{202420326@mail.sdu.edu.cn} (Y. Chang),
\texttt{shuowei@mail.sdu.edu.cn} (S. Wei), and
\texttt{yanj@sdu.edu.cn} (J. Yan). Supported by National Natural Science Foundation of China (No.12571373).}\and
Louis DeBiasio\thanks{Department of Mathematics, Miami University,
Oxford, OH, USA. E-mail: \texttt{debiasld@miamioh.edu}. Research supported
in part by AMS--Simons Research Enhancement Grants
for PUI Faculty GR001015.}\and
H.~A.~Kierstead\thanks{Arizona State University,
Tempe, AZ, USA. E-mail: \texttt{kierstead@asu.edu}.}\and
Shuo Wei\footnotemark[1]
\and Jin Yan\footnotemark[1]}
\date{}

\begin{document}

\maketitle

\begin{abstract}
The Hajnal--Szemer\'edi theorem~\cite{HS} states that every graph with
$\Delta(G)\le k-1$ admits an equitable $k$-coloring.  Using a reduction implicit in Aboulker, Oijid, Petit, Rocton, and
Simon~\cite{AOPRS} (which itself builds on the literature surrounding feedback arc sets), we prove that this theorem is equivalent to each of four digraph analogues guaranteeing equitable acyclic $k$-colorings
under various degree conditions. A minor modification of the proof
method in~\cite{AOPRS} makes the construction algorithmic, so the
equitable acyclic coloring guaranteed by these statements
can be found in polynomial time. The same reduction shows that the
Ore-type Hajnal--Szemer\'edi theorem is equivalent to its four digraph
analogues.

These equivalences settle the conjecture of
Treglown~\cite{Treglown2015} that every $n$-vertex digraph with
$r\mid n$ and $\delta^{\max}(D)\ge(1-\frac{1}{r})n$ contains a $T_r$-factor,
answer a question of Molla and Treglown~\cite{MollaTreglown} by
determining the exact Ore-type threshold for $T_r$-factors, and recover the directed Hajnal--Szemer\'edi theorems of Czygrinow, DeBiasio,
Kierstead, and Molla~\cite{Czygrinow2015} as immediate corollaries of
this reduction.

We also discuss related problems for oriented graphs which do not follow immediately from the reduction mentioned above.
\end{abstract}

\keywords{Digraphs; transitive tournaments; factors; equitable coloring.} 

\subjclass{05C35; 05C20; 05C70.} 

\section{Introduction}
A digraph $D$ is a pair $(V(D), E(D))$ such that $V(D)$ is finite and $E(D)$ is an irreflexive relation on $V(D)$.  A graph $G$ is a digraph where $E(G)$ is symmetric, and an oriented graph $D$ is a digraph where $E(D)$ is anti-symmetric.

The study of graph factors has long been a central topic in graph theory.
Given two (di)graphs $H$ and $G$, an $H$-tiling in $G$ is a collection of vertex-disjoint copies of $H$ in $G$. An $H$-factor in $G$ is an $H$-tiling that covers all vertices of $G$. 
Corr\'adi and Hajnal~\cite{CH} proved that every $n$-vertex graph $G$ with $3\mid n$  and minimum degree at least $\frac{2n}{3}$ contains a triangle factor. Solving a problem of Erd\H{o}s \cite{E}, Hajnal and Szemer\'edi extended this result to general cliques: for every integer $r\ge 2$, if $r \mid n$ and $G$ is an $n$-vertex graph with minimum degree $\delta(G)\ge (1-\frac{1}{r})n$, then $G$ contains a $K_r$-factor.
It also has an equivalent formulation in terms of equitable coloring.
By a $k$-partition of a set $V$, we mean an indexed $k$-tuple
$(V_1,\ldots,V_k)$ of pairwise disjoint, possibly empty subsets whose
union is $V$.  Such a $k$-partition is \emph{equitable} if
$\bigl||V_i|-|V_j|\bigr|\le 1$ for all $i,j\in [k]$.  An
\emph{equitable} $k$-coloring of a graph $G$ is an equitable
$k$-partition $(V_1,\dots,V_k)$ of $V(G)$ such that $V_i$ is an
independent set for all $i\in [k]$.

\begin{theorem}[Hajnal and Szemer\'edi \cite{HS}]\label{th1}
Let  $k\in\mathbb{N}$ and let $G$ be a graph. If $\Delta(G)\le k-1$, then $G$ has an equitable $k$-coloring.
\end{theorem}

Theorem~\ref{th1} has a polynomial-time algorithmic counterpart.

\begin{theorem}[Kierstead, Kostochka, Mydlarz and
Szemer\'edi~\cite{KKMS}]\label{thm:kkms}
For every positive integer $k$ and every $n$-vertex graph $G$ with
$\Delta(G)\le k-1$, an equitable $k$-coloring of $G$ can be found
in time $O(kn^2)$.
\end{theorem}

The hypothesis of Theorem~\ref{th1} was later weakened to an Ore-type
condition by Kierstead and Kostochka~\cite{KiersteadKostochka}.

\begin{theorem}[Kierstead  and Kostochka~\cite{KiersteadKostochka}]\label{thm:KK}
Let $k\ge2$ and let $G$ be a graph. If $d_G(x)+d_G(y)\le2k-1$ for all $xy\in E(G)$, then $G$ has an equitable $k$-coloring.
\end{theorem}

The Hajnal--Szemer\'edi theorem was first generalized to acyclic colorings in digraphs by Czygrinow, DeBiasio, Kierstead, and Molla~\cite{Czygrinow2015} and Treglown~\cite{Treglown2015}. An \emph{equitable acyclic $k$-coloring} of $D$ is an equitable
$k$-partition $(V_1,\ldots,V_k)$ of $V(D)$ such that every $D[V_i]$ is acyclic.  
However, before we can state their results formally, we must introduce some notation.

Given $v\in V(D)$, let
$N^+(v):=\{w\in V(D):(v,w)\in E(D)\}$ and $N^-(v):=\{u\in V(D):(u,v)\in E(D)\}$ and set $d^\pm(v)=|N^\pm(v)|$.
Write $d(v)=d^+(v)+d^-(v)$. Let $\Delta^\pm(D)=\max\{d^\pm(v):v\in V(D)\}$ and $\Delta(D)=\max\{d(v):v\in V(D)\}$.  Analogously, let $\delta^\pm(D)=\min\{d^\pm(v):v\in V(D)\}$ and $\delta(D)=\min\{d(v):v\in V(D)\}$.

Define the \emph{mindegree} and \emph{maxdegree} of a vertex $v$ by\footnote{As far as we are aware, the terms \emph{mindegree} and \emph{maxdegree} were introduced by Aboulker and Aubian~\cite{AA} and appear again in later works such as \cite{AOPRS} and \cite{PA}.}
$\dmin(v):=\min\{d^+(v),d^-(v)\}$
and $\dmax(v):=\max\{d^+(v),d^-(v)\}$. Define the \emph{maximum mindegree} and \emph{maximum maxdegree} as
$\Dmin(D):=\max\{\dmin(v):v\in V(D)\}$ and 
$\Dmax(D):=\max\{\dmax(v):v\in V(D)\}$,
respectively.  We define $\delmin$ and $\delmax$ analogously.

Our main results extend the (Ore-type) Hajnal--Szemer\'edi theorem to the setting of digraphs. More elegantly, we further show that all of these results are equivalent to the graph version of the (Ore-type) Hajnal--Szemer\'edi theorem, namely Theorems~\ref{th1} and~\ref{thm:KK}.

\begin{theorem}\label{HS}
For every positive integer $k$ and every digraph $D$, the following statements are true. Moreover, all statements are equivalent to Theorem~\ref{th1}. Furthermore, the equitable acyclic coloring can be found in polynomial time.
\begin{enumerate}[
    label=$\mathbf{HS}\,\arabic*:$,
    ref=HS\,\arabic*,
    labelsep=0.4em,
    leftmargin=2.2em,
    align=left,
    itemsep=0.1em
]
\item\label{HS1} If $\Delta^{\min}(D)\le k-1$, then $D$ has an equitable acyclic $k$-coloring.
\item\label{HS2} For each fixed $\bullet\in\{+,-\}$, if $\Delta^\bullet(D)\le k-1$, then $D$ has an equitable acyclic $k$-coloring.
\item\label{HS3} If $\Delta(D)\le 2k-1$, then $D$ has an equitable acyclic $k$-coloring.
\item\label{HS4} If $\Delta^{\max}(D)\le k-1$, then $D$ has an equitable acyclic $k$-coloring.
\end{enumerate}
\end{theorem}

Indeed, Theorem~\ref{HS} can be rephrased in the complementary language of factors.
Let $T_r$ denote the transitive tournament on $r$ vertices, and for a digraph
$D$, let $\bar D$ denote its complement,  for distinct $x,y\in V(D)$, define that
$xy\in E(\bar D)$ if and only if $xy\notin E(D)$. For $r\mid n$ and $k=\frac{n}{r}$, we have $\Dmin(D)\le k-1$ if and only if $\delta^{\max}(\bar D)\ge (1-\frac{1}{r})n$. Moreover,  $D$ has an equitable acyclic $k$-coloring if and only if $\bar D$ contains a $T_r$-factor.
Czygrinow, DeBiasio, Kierstead, and Molla~\cite{Czygrinow2015} proved \ref{HS3}
and noted that minor modifications of their proof also give \ref{HS2}.  Around the same time, Treglown \cite{Treglown2015} proposed \ref{HS1} as the following  conjecture and asymptotically solved this conjecture for sufficiently large $n$.

\begin{conjecture}[Treglown~\cite{Treglown2015}]\label{con:treglown}
Let $n, r \in \mathbb{N}$ with $r \mid n$ and $r \ge 3$, and let $G$ be an $n$-vertex digraph. If $\delta^{\max}(G)\ge(1-\frac{1}{r})n$, then $G$ contains a $T_r$-factor.

\end{conjecture}

Our second main result extends Theorem~\ref{thm:KK} to the setting of digraphs.

\begin{theorem}\label{HSO}
For every integer $k\ge2$ and every digraph $D$, the following statements hold.
Moreover, each statement is equivalent to Theorem~\ref{thm:KK}.

\begin{enumerate}[
    label=$\mathbf{OHS}\,\arabic*:$,
    ref=OHS\,\arabic*,
    labelsep=0.4em,
    leftmargin=2.2em,
    align=left,
    itemsep=0.1em
]
\item\label{OHS1}
If $\dmin(x)+\dmin(y)\le 2k-1$ for every $(x,y)\in E(D)$, then $D$ has an equitable acyclic $k$-coloring.

\item\label{OHS2}
For each fixed choice $\circ,\bullet\in\{+,-\}$, if $d^{\circ}(x)+d^{\bullet}(y)\le 2k-1$ for every $(x,y)\in E(D)$, then $D$ has an equitable acyclic $k$-coloring.

\item\label{OHS3}
If $d(x)+d(y)\le 4k-1$ for every $(x,y)\in E(D)$, then $D$ has an equitable acyclic $k$-coloring.

\item\label{OHS4}
If $\dmax(x)+\dmax(y)\le 2k-1$ for every $(x,y)\in E(D)$, then $D$ has an equitable acyclic $k$-coloring.
\end{enumerate}
\end{theorem}

In the factor language, \ref{OHS2} takes the following form: for $r\mid n$, $k=\frac{n}{r}$, and each fixed choice $\circ,\bullet\in\{+,-\}$, every $n$-vertex digraph $G$ with $d^\circ(x)+d^\bullet(y)\ge2(1-\frac{1}{r})n-1$ for every ordered pair  of distinct vertices $x,y$ with $xy\notin E(G)$ contains a $T_r$-factor. Molla and Treglown~\cite[Theorems~1.9 and~1.10]{MollaTreglown} proved this exactly for $r=3$, i.e.\ for $k=\frac{n}{3}$, with threshold $\frac{4n}{3}-1$, and proved an asymptotic version for every fixed $r$, namely with $2(1-\frac{1}{r})n+o(n)$ in place of $2(1-\frac{1}{r})n-1$. Molla and Treglown further proposed \ref{OHS2} as a problem, and thus Theorem \ref{HSO} settles their problem.

\begin{problem}[Molla and Treglown \cite{MollaTreglown}]\label{pro1}
    Determine the exact value of the Ore-type threshold for forcing a $T_r$-factor in digraphs.
\end{problem}

\subsection{Notation}\label{sec:pre}
Let $D$ be a digraph, the number of vertices $ |D| := |V(D)| $ is called the \textit{order} of $ D $, and the number of arcs $ e(D) := |E(D)| $ is called the \textit{size} of $ D $. Given two vertices $ x $ and $ y $ of $ D $, we write $ xy $ for the arc directed from $ x $ to $ y $. A  \emph{directed cycle} is a cycle $v_1 v_2 \dots v_\ell v_1$, where $\ell\ge2$, such that $v_i v_{i+1} \in E(D)$ for every $1 \le i \le \ell$ (indices taken modulo $\ell$).
A digraph is \emph{acyclic} if it contains no directed cycle.
Given $X\subseteq V(D)$, let $D[X]$ denote the digraph induced by $X$.
For $k \in \mathbb{N}$, we write $[k] = \{1, 2, \dots, k\}$.

Given a digraph $D$, a set $F\subseteq E(D)$ is a \emph{feedback arc set} of $D$
if $D-F:=(V(D),E(D)\setminus F)$ is acyclic.  Let $G_{F}$ be the graph on $V(D)$ with $ E(G_{F}):=\{\{u,v\}:uv\in F\text{ or }vu\in F\}.$
If $I$ is an independent set of $G_{F}$, then no arc of $F$ has both end-vertices in $I$. Consequently $D[I]\subseteq D-F$, and hence $D[I]$ is acyclic.

Let $\prec$ be a linear order on $V(D)$.  An arc $uv$ is a
\emph{backedge} of $\prec$ if $v\prec u$.  The \emph{backedge graph}
$B_\prec(D)$ is the graph on $V(D)$ in which $uv$ is an edge whenever
at least one of $uv$ and $vu$ is a backedge. A \emph{topological order} of a digraph $D$ is a linear order $\prec$
on $V(D)$ such that $u\prec v$ for every arc $uv\in E(D)$.

\section{Reduction to the graph case}\label{sec:reduction}

\subsection{Feedback arc sets and backedge graph}
This feedback-arc-set or backedge graph viewpoint has been
used in several works on tournament and digraph coloring.  In particular, Davot, Isenmann, Roy, and Thiebaut~\cite{DIRT} defined the \emph{degreewidth} of a digraph\footnote{Technically, they only defined this for tournaments} $D$ as $\overrightarrow\Delta(D):=\min_\prec\Delta(B_\prec(D))$.  Aboulker, Oijid, Petit, Rocton, and
Simon~\cite[Lemma~3.1]{AOPRS} developed the backedge-graph framework for general
digraphs and proved that
\begin{equation}\label{eq:d-bound}
  \overrightarrow\Delta(D)\le\Dmin(D).
\end{equation}
The following lemma is implicit in the proof of \cite[Lemma 2.1]{AOPRS} and relates these two concepts.  

\begin{lemma}\label{lem:feedback-backedge}
Let $F$ be an inclusion-minimal feedback arc set of a digraph $D$.
If $\prec$ is a topological order of $D-F$, then $G_F=B_\prec(D)$.
\end{lemma}

\begin{proof}
Every arc not in  $F$ is forward in $\prec$.  If an arc $e\in F$
were also forward, then $\prec$ would remain a topological order after
adding $e$ back to $D-F$.  Equivalently, $F\setminus\{e\}$ would
still be a feedback arc set, contrary to the inclusion-minimality of
$F$.  Thus the arcs of $F$ are precisely the backedges of $D$ with
respect to $\prec$, which proves the equality.
\end{proof}

For a feedback arc set $F$ of $D$ and a vertex $v\in V(D)$, let
$F(v)$ denote the set of arcs of $F$ incident with $v$.  For
$\sigma\in\{+,-\}$, let $A_D^\sigma(v)$ denote the set of arcs leaving
or entering $v$, respectively, and set
\[
F_v^\sigma=(F\setminus F(v))\cup A_D^\sigma(v).
\]
The set $F_v^\sigma$ is again a feedback arc set: every directed cycle
avoiding $v$ meets $F\setminus F(v)$, while every directed cycle
containing $v$ meets $A_D^\sigma(v)$.

The following is a feedback-arc-set formulation of \eqref{eq:d-bound}.
The local moves in the argument below are equivalent to those
in~\cite{AOPRS}; we include the proof in order to make the
polynomial-time construction explicit.

\begin{lemma}\label{lem:feedback}
Every digraph $D$ has a feedback arc set $F\subseteq E(D)$ such that $d_{G_{F}}(v)\leq\min\{d_D^+(v),d_D^-(v)\}$ for every $v\in V(D)$.
Furthermore, such an $F$ can be found in polynomial time.
\end{lemma}

\begin{proof}
Begin with $F:=E(D)$, which is clearly a feedback arc set.  Note that
$d_{G_F}(v)\le |F(v)|$, since two oppositely directed arcs in $F(v)$
give only one edge of $G_F$.  If $|F(v)|>d_D^+(v)$, then
\[
  |F_v^+|=|F|-|F(v)|+d_D^+(v)<|F|.
\]
Similarly, if $|F(v)|>d_D^-(v)$, then
$|F_v^-|=|F|-|F(v)|+d_D^-(v)<|F|$.

Starting with $F=E(D)$, we repeatedly perform either applicable
replacement $F:=F_v^+$ or $F:=F_v^-$.  Each replacement preserves the
property of being a feedback arc set and strictly decreases $|F|$, so
the process terminates.  At termination, for every $v\in V(D)$ we have
\[
  |F(v)|\le d_D^+(v)
  \quad\text{and}\quad
  |F(v)|\le d_D^-(v).
\]
Consequently $d_{G_{F}}(v)\le |F(v)|\le\min\{d_D^+(v),d_D^-(v)\}.$

Set $n:=|V(D)|$ and $m:=|E(D)|$.  The algorithm updates the current feedback arc set at most $m$ times, and each update requires scanning at most $2(n-1)$ incident arcs; thus the algorithm uses at most $2m(n-1)$ incident-arc scans.  So the running time is $O(nm)$.
\end{proof}

\begin{remark}\label{rem:fas-backedges}
The feedback arc set produced above need not have minimum cardinality,
or even be inclusion-minimal.  However, we may delete arcs from $F$
until an inclusion-minimal feedback arc set $F_1$ remains, and this can only decrease the degrees of the associated graph.  Thus by Lemma~\ref{lem:feedback-backedge} it follows that Lemma~\ref{lem:feedback} implies \eqref{eq:d-bound}.
\end{remark}

\subsection{Proof of the main theorems}

We  make the following two simple lemmas regarding the degree
conditions.

\begin{lemma}\label{obs:deg-con}
    Let $D$ be a digraph and let $k$ be a positive integer.
\begin{itemize}
    \item[(i)] If $\Delta(D)\le 2k-1$ or $\Delta^\circ(D)\le k-1$ for some $\circ\in \{+,-\}$, then $\Delta^{\min}(D)\le k-1$.
    \item[(ii)] If $\Delta^{\max}(D)\le k-1$, then $\Delta(D)\le 2k-2$ and $\Delta^\bullet(D)\le k-1$ for all $\bullet\in \{+,-\}$.
\end{itemize}
\end{lemma}

\begin{lemma}\label{obs:ore-deg-con}
Let $D$ be a digraph, let $k$ be a positive integer, and let $x,y$ be
adjacent vertices of $D$.

\begin{itemize}
    \item[(i)]  If $d(x)+d(y)\le 4k-1$, or if $d^\circ(x)+d^\bullet(y)\le 2k-1$ for some $\circ,\bullet\in\{+,-\}$, then $d^{\min}(x)+d^{\min}(y)\le 2k-1$.
    \item[(ii)] If $d^{\max}(x)+d^{\max}(y)\le 2k-1$, then $d(x)+d(y)\le 4k-2$ and $d^\circ(x)+d^\bullet(y)\le 2k-1$ for all $\circ,\bullet\in\{+,-\}$.
\end{itemize}

\end{lemma}

\begin{proof}[Proof of Theorem~\ref{HS}]
    By Lemma~\ref{obs:deg-con}(i), \ref{HS1} implies \ref{HS2} and \ref{HS3}.  Likewise, by Lemma~\ref{obs:deg-con}(ii), each of \ref{HS2} and \ref{HS3} implies \ref{HS4}.  Since a graph is just a symmetric digraph and an acyclic coloring of a symmetric digraph is a proper coloring,
\ref{HS4} implies Theorem~\ref{th1}.  Moreover, \ref{HS2} is
equivalent for $\bullet\in\{+,-\}$, by reversing the direction of every edge.

Finally, we show that Theorem~\ref{th1} implies \ref{HS1}.  Let $D$ be a digraph with
$\Dmin(D)\le k-1$, let $F$ be the feedback arc set given by
Lemma~\ref{lem:feedback}, and set $G:=G_{F}$.  Hence
\[
  \Delta(G)\le\Dmin(D)\le k-1.
\]
By Theorem~\ref{th1}, the graph $G$ has an equitable $k$-coloring, and by
Theorem~\ref{thm:kkms} such a coloring can be found in $O(kn^2)$ time.  Every color
class is an independent set of $G$ and hence
induces an acyclic subdigraph of $D$; thus this coloring is an equitable
acyclic $k$-coloring of $D$.  Together with the $O(nm)$-time construction of
$F$ in Lemma~\ref{lem:feedback}, the equitable acyclic $k$-coloring of $D$ can be
found in $O(nm+kn^2)$ time.

This proves that Theorem~\ref{th1}, \ref{HS1}, \ref{HS2}, \ref{HS3}, and \ref{HS4} are all equivalent.
\end{proof}

\begin{proof}[Proof of Theorem~\ref{HSO}]
By Lemma~\ref{obs:ore-deg-con}, \ref{OHS1} implies \ref{OHS2} and
\ref{OHS3}, while each of those statements implies \ref{OHS4}.
Moreover, \ref{OHS4} implies Theorem~\ref{thm:KK} by viewing a graph as a symmetric digraph. 

We now prove that Theorem~\ref{thm:KK} implies \ref{OHS1}.  Let $D$ be a digraph satisfying the hypothesis of
\ref{OHS1}.  Let $F$ be the feedback arc set in $D$ given by
Lemma~\ref{lem:feedback}, and set $G:=G_{F}$.  If
$xy\in E(G)$, then at least one of $xy$ and $yx$ is an arc of $F$, hence an arc of $D$.
  Since $\dmin_D(x)+\dmin_D(y)$ is symmetric in $x$ and $y$, the
hypothesis therefore gives
\[
  d_G(x)+d_G(y)
  \leq \dmin_D(x)+\dmin_D(y)
  \leq 2k-1.
\]
Thus by Theorem~\ref{thm:KK}, $G$ has an equitable $k$-coloring.  As before, every color class is independent in $G$ and therefore induces an acyclic subdigraph of $D$.  This proves \ref{OHS1}.
Thus all  statements are equivalent. 
\end{proof}

\section{Oriented graphs}

For oriented graphs, much less is known regarding equitable acyclic
colorings.  For $k\ge2$, let $f(k)$ be the largest integer $d$ such
that every oriented graph $D$ with $\Dmin(D)\le d$ has an equitable
acyclic $k$-coloring.  Theorem~\ref{HS1} gives $f(k)\ge k-1$.
Picasarri-Arrieta~\cite[Corollary~8]{PA} proved that every oriented
graph $D$ with $\Dmin(D)\le k$ has an ordinary acyclic $k$-coloring,
which motivates the following question.

\begin{question}\label{q:oriented}
Determine $f(k)$.  In particular, is $f(k)\ge k$ for every $k\ge2$?
\end{question}

For ordinary acyclic coloring, recent Reed-type and
Borodin--Kostochka-type results show that Picasarri-Arrieta's bound is
not expected to be sharp.  Writing $\vec\chi(D)$ for the dichromatic
number and
$\widetilde\Delta(D):=\max_{v\in V(D)}\sqrt{d^+(v)d^-(v)}$,
Kawarabayashi and Picasarri-Arrieta proved that every oriented graph
$D$ satisfies
$\vec\chi(D)\le(\frac{\sqrt{2}}{2})\widetilde\Delta(D)+2$, and more generally
proved a directed analogue of Reed's convex-combination
theorem~\cite[Theorems~1.5 and~1.7]{KawarabayashiPicasarri}.
They also conjectured a substantially stronger bound in terms of
$\Dmin$: there exists $\varepsilon>0$ such that every oriented graph
$D$ has an acyclic
$\ceil{(1-\varepsilon)(\Dmin(D)+1)+2\varepsilon}$-coloring, and they
suggest that one may take $\varepsilon=\frac{1}{2}$; see
\cite[Conjecture~8.2]{KawarabayashiPicasarri}.

More recently, Harutyunyan, Kawarabayashi, Picasarri-Arrieta and
Puig i Surroca proved a directed analogue of Reed's large-degree case
of the Borodin--Kostochka conjecture.  In particular, for every
sufficiently large $k$, every oriented graph $D$ with $\Dmin(D)\le k$
has an acyclic $(k-1)$-coloring~\cite[Corollary~13]{HKPP}.  Thus, at
least for large $k$, the particular inequality $f(k)\ge k$ is
specifically about balancing the color classes: an ordinary acyclic
coloring exists even with one fewer color, but the known methods do
not control the class sizes.

For $k=2$, Theorem~\ref{thm:clw-even} below gives $f(2)\ge2$.
On the other hand, the Paley tournament on seven vertices has
$\Dmin=3$ and contains no transitive subtournament on four vertices,
so it has no acyclic $2$-coloring.  Hence $f(2)=2$.  We have no evidence to suggest
that $f(k)=k$ in general.  Indeed,
Observation~\ref{obs:oriented-ramsey-range}(i) shows that the
tournament-restricted analogue of $f(k)$ is at least
$(\frac{1}{4}-o(1))k\log_2 k$, while part~(ii) gives
$f(k)\le(1+o(1))k\log_2 k$.

\subsection{Complementary problems in the oriented case}

For oriented graphs, the equitable-coloring and transitive-tournament factor
formulations lead to different problems, since the complement
of an oriented graph is a semicomplete digraph.  Thus, when $n=kr$, an
equitable acyclic $k$-coloring of an oriented graph $D$ corresponds to a
$T_r$-factor in the semicomplete digraph $\bar D$; conversely, a
$T_r$-factor in an oriented graph corresponds to an equitable acyclic
$k$-coloring of its semicomplete complement.  Nevertheless, there is 
a useful tournament formulation of the $T_r$-factor problem for oriented graphs. 
Let $G$ be an oriented graph on $n=kr$ vertices,  complete $G$ 
arbitrarily to a tournament $T$, and let $M$ be the graph on $V(G)$
whose edges are the missing pairs of $G$.  For every vertex $v$, we have
$d_M(v)=n-1-d_G(v)$, and hence $\Delta(M)=n-1-\delta(G)$.  
A partition of $V(G)$ into $k$ sets of
size $r$ is a $T_r$-factor in $G$ if and only if every part is
independent in $M$ and transitive in $T$.  Equivalently, $G$ has a
$T_r$-factor if and only if there is a linear order $\prec$ for which
$M\cup B_\prec(T)$ has an equitable $k$-coloring.  Thus the
minimum-total-degree $T_r$-factor problem for oriented graphs is a
conflict-constrained version of equitable acyclic coloring in
tournaments: the edges of $M$ are pairs that are forbidden from
receiving the same color.

For example, in this language the result of DeBiasio, Lo, Molla and
Treglown~\cite{DLMT} says that, for every $\varepsilon>0$ and all
sufficiently large $n\in4\mathbb{N}$, if $T$ is an $n$-vertex
tournament and $M$ is a graph on $V(T)$ with
$\Delta(M)\le(\frac{1}{12}-\varepsilon)n$, then $V(T)$ can be partitioned into
four-vertex sets that are independent in $M$ and transitive in $T$.
This is asymptotically sharp: there are examples with
$\Delta(M)=\floor{\frac{n}{12}}$ for which no such partition exists.  Thus
$\frac{1}{12}$ is the asymptotically largest maximum-degree proportion of an
arbitrary conflict graph that a tournament $T_4$-factor can always
tolerate.  This formulation does not directly give a stronger bound
on $f(k)$, since missing pairs may lie within an acyclic color class in
Question~\ref{q:oriented}; nevertheless, it identifies the common
backedge-coloring structure underlying the two problems.

The \emph{oriented tiling Ramsey number}
$\overrightarrow{\mathrm{tr}}(r)$ is the smallest integer $N$ such that
every tournament on $N$ vertices has a $T_r$-factor.  The same conclusion
then holds whenever $n\ge\overrightarrow{\mathrm{tr}}(r)$ and $r\mid n$.
This parameter was defined by Treglown~\cite{Treglown2012}, who attributes
its existence to Erd\H{o}s; see Reid~\cite{Reid}.  The
Hajnal--Szemer\'edi theorem implies that every $n$-vertex oriented graph
$D$ with $r\mid n$ and
$\delta(D)\ge(1-\frac{1}{\overrightarrow{\mathrm{tr}}(r)})n$ has a
$T_r$-factor; see~\cite[Proposition~1.9]{Treglown2012}
and~\cite{Yuster}.  Since $\overrightarrow{\mathrm{tr}}(3)=6$, this gives
the sharp bound $\frac{5n}{6}$ for a $T_3$-factor.  For $r=4$, however,
$\overrightarrow{\mathrm{tr}}(4)=16$, so this argument gives $\frac{15n}{16}$,
whereas the conflict-constrained result above gives the asymptotically
sharp threshold $(\frac{11}{12}+o(1))n$.

For equitable coloring, completing an oriented graph to a tournament
shows that no degree condition is needed when $k\mid n$ and
$n\ge\overrightarrow{\mathrm{tr}}(\frac{n}{k})$.  Hunter observed that the
method of Buci\'c and Sudakov can be adapted to show
$\overrightarrow{\mathrm{tr}}(r)\le Cr^\ell2^r$ for absolute constants
$C,\ell>0$; see~\cite[p.~20]{BucicSudakov}.
\begin{observation}\label{obs:oriented-ramsey-range}
The following statements hold.
\begin{enumerate}[label=\textup{(\roman*)}]
\item For every $\varepsilon>0$ and all sufficiently large $n$, every
$n$-vertex oriented graph has an equitable acyclic $k$-coloring whenever
$k\ge(1+\varepsilon)\frac{n}{\log_2 n}$.  In particular, for every
$0<c<\frac{1}{4}$ and all sufficiently large $k$, every tournament $T$ with
$\Dmin(T)\le c k\log_2 k$ has an equitable acyclic $k$-coloring.
\item For every $0<\varepsilon<\frac{1}{2}$ and all sufficiently large $n$, there
exists an $n$-vertex tournament with no acyclic $k$-coloring whenever
$k\le(\frac{1}{2}-\varepsilon)\frac{n}{\log_2 n}$.  In particular, for every
$\eta>0$ and all sufficiently large $k$, there exists an oriented graph
$D$ with $\Dmin(D)\le(1+\eta)k\log_2 k$ and no acyclic $k$-coloring.
\end{enumerate}
\end{observation}

\begin{proof}
For (i), complete the oriented graph to a tournament.  The case $k>n$
is immediate, so suppose $k\le n$ and write $n=kr+s$, where
$0\le s<k$.  An equitable coloring requires $k-s$ classes of size $r$
and $s$ classes of size $r+1$.  The assumed lower bound on $k$ gives
$r+1\le\frac{\log_2 n}{1+\varepsilon}+1$.  Consequently,
$C(r+1)^\ell2^{r+1}=o(n)$, so, for sufficiently large $n$,
$\overrightarrow{\mathrm{tr}}(r),
\overrightarrow{\mathrm{tr}}(r+1)<\frac{n}{2}$.
Of the vertices assigned to classes of size $r$ and those assigned to
classes of size $r+1$, one group has total size at most $\frac{n}{2}$.
Greedily remove the required transitive subtournaments for that group.
Before each removal, at least $\frac{n}{2}$ vertices remain, so the relevant
oriented tiling Ramsey bound guarantees another copy.  The remaining
tournament has order at least $\frac{n}{2}$, is divisible by the other class
size, and hence has a factor of transitive tournaments of that size.
These transitive sets are acyclic in the original oriented graph and
form an equitable $k$-coloring.

For the final assertion in (i), let $d=\Dmin(T)$ and set
$A:=\{v:d_T^+(v)\le d\}$ and $B:=\{v:d_T^-(v)\le d\}$.  Every vertex
belongs to $A\cup B$.  The subtournament $T[A]$ has a vertex of
outdegree at least $\frac{|A|-1}{2}$ within $A$, so $|A|\le2d+1$; similarly,
$|B|\le2d+1$.  Hence $|T|\le4d+2$.  The case $k>|T|$ is immediate,
so suppose $k\le|T|$.  If
$d\le c k\log_2 k$ with $c<\frac{1}{4}$, then
$\frac{|T|}{\log_2|T|}\le(4c+o(1))k$.  Thus, for some fixed
$\varepsilon>0$ and all sufficiently large $k$, the hypothesis in
(i) holds.

For (ii), the random-tournament construction of Erd\H{o}s and
Moser~\cite{ErdosMoser} gives, for all sufficiently large $n$, an
$n$-vertex tournament whose largest transitive subtournament has order
at most $(2+o(1))\log_2 n$.  In any acyclic $k$-coloring of a tournament,
each color class is transitive and some class has order at least $\frac{n}{k}$.
If $k\le(\frac{1}{2}-\varepsilon)\frac{n}{\log_2 n}$, then
\[
  \frac nk\ge\frac{2}{1-2\varepsilon}\log_2 n
  >(2+o(1))\log_2 n,
\]
a contradiction.

For the final assertion, fix $\eta>0$, set
$\varepsilon=\frac{\eta}{8+4\eta}$, and, for sufficiently large $k$, let
$n=\ceil{2(1+\frac{\eta}{2})k\log_2 k}$.  Thus
$(\frac{1}{2}-\varepsilon)\frac{n}{\log_2 n}>k$, so (ii) supplies an $n$-vertex
tournament $T$ with no acyclic $k$-coloring.  Since every $n$-vertex
tournament satisfies $\Dmin(T)\le\frac{n-1}{2}$, we have
$\Dmin(T)\le(1+\eta)k\log_2 k$ for sufficiently large $k$.
\end{proof}

These bounds leave open what happens in the intermediate range.


\subsection{Consequences of the Chen--Lih--Wu conjecture}\label{app:clw}

We prove two conditional results related to
Question~\ref{q:oriented}.  For $k\ge2$, let $\mathrm{CLW}_k$ denote the following assertion from the
Chen--Lih--Wu conjecture~\cite{ChenLihWu}:
\begin{quote}
Every connected graph $J$ with $\Delta(J)\le k$ has an equitable
$k$-coloring, unless $J=K_{k+1}$, or $k=2$ and $J$ is an odd cycle,
or $k$ is odd and $J=K_{k,k}$.
\end{quote}
The assertion $\mathrm{CLW}_2$ is elementary: every connected graph
of maximum degree at most two is a path or a cycle, and every such
graph other than an odd cycle has an equitable $2$-coloring.

\begin{theorem}\label{thm:clw-even}
Let $k\ge2$ be even.  If $\mathrm{CLW}_k$ holds, 
then every oriented graph $D$ with $\Dmin(D)\le k$ has an equitable
acyclic $k$-coloring.
\end{theorem}

\begin{theorem}\label{thm:clw-odd}
Let $k\ge3$ be odd.  If $\mathrm{CLW}_k$ holds, then every oriented graph $D$ satisfying either $\Delta^+(D)\le k$ or $\Delta^-(D)\le k$ has an
equitable acyclic $k$-coloring.
\end{theorem}

Since $\mathrm{CLW}_2$ is elementary, $\mathrm{CLW}_3$ was proved by
Chen, Lih and Wu~\cite{ChenLihWu}, and $\mathrm{CLW}_4$ by Kierstead
and Kostochka~\cite{KKfour}, these results give unconditional
conclusions for $k\leq4$.

We first record two elementary coloring observations.  Equitable
$k$-colorings of disjoint graphs can be combined equitably by
permuting color labels.  Indeed, if the current class sizes are $a$ or
$a+1$ and the next graph has classes of sizes $b$ or $b+1$, assign its
larger classes to the currently smaller classes first.  The resulting
sizes again differ by at most one, so induction gives the claim.

Given a digraph $D$ and feedback arc set $F$, recall that an equitable coloring of $G_F$ gives an equitable
acyclic coloring of $D$.  Let $F$ be a minimum feedback arc set of an
oriented graph $D$.  By the definitions in
Section~\ref{sec:reduction}, $F_v^+$ and $F_v^-$ are feedback arc sets.
Since $D$ is oriented, $|F(v)|=d_{G_F}(v)$, so comparison with these
two feedback arc sets gives
\begin{equation}\label{eq:local}
d_{G_F}(v)\le\min\{d_D^+(v),d_D^-(v)\}
\qquad\text{for every }v\in V(D).
\end{equation}

Since $F$ has minimum cardinality, it is inclusion-minimal.  The proof
of Lemma~\ref{lem:feedback-backedge} shows that every arc of $F$ points
backward in a topological order of $D-F$.  The reverse order is
therefore a topological order of the digraph $(V(D),F)$, so this
digraph is acyclic.

\begin{lemma}\label{lem:app-no-clique}
For every $k\ge2$ and every oriented graph $D$ with $\Dmin(D)\le k$,
there is a feedback arc set $F$ satisfying \eqref{eq:local} such that $G_F$ has no
$K_{k+1}$ component, and if $k=2$, $G_F$ has no odd cycle.
\end{lemma}

\begin{proof}
First suppose $k\ge3$.  Choose a minimum feedback arc set $F$ for
which $G=G_F$ has the fewest $K_{k+1}$ components.  Suppose $C$ is
such a component.  Its orientation in $F$ is acyclic, hence
transitive.  Choose a vertex $v$ other than its source and sink.
Since $d_G(v)=k$, \eqref{eq:local} and
$\Dmin(D)\le k$ give
$\min\{d_D^+(v),d_D^-(v)\}=k$.  Reversing all arcs if necessary,
we may assume $d_D^+(v)=k$.

List the vertices of $C$ as $v_0,\ldots,v_k$ so that its $F$-arcs
are $v_av_b$ with $a<b$, and write $v=v_i$, where
$1\le i\le k-1$.  The outgoing arcs of $v_i$ inside $C$ are
$v_iv_{i+1},\ldots,v_iv_k$.  Since $d_D^+(v_i)=k$, its other
outgoing arcs are $v_iw_1,\ldots,v_iw_i$ for distinct vertices
$w_1,\ldots,w_i$ outside $C$.  Define
\[
F'=\bigl(F\setminus\{v_0v_i,\ldots,v_{i-1}v_i\}\bigr)
   \cup\{v_iw_1,\ldots,v_iw_i\}.
\]
Thus $F'=F_{v_i}^+$, and we deleted and inserted the same number of
arcs.  Hence $F'$ is another minimum feedback arc set.

In $G_{F'}$, the vertices of $C-v_i$ still form a clique, and the
edge $v_iv_k$ still joins $v_i$ to that clique.  Each new edge
$v_iw_j$ attaches the old component containing $w_j$ to these
vertices, while every other component is unchanged.  The vertex
$v_0$ has lost its edge to $v_i$ and gained none, so its new degree
is $k-1$.  Thus the new component is not $K_{k+1}$.  We have removed
at least one $K_{k+1}$ component and created none, contrary to the
choice of $F$.

For $k=2$, choose a minimum feedback arc set $F$ for which $G_F$ has
the fewest odd-cycle components.  By
\eqref{eq:local}, $\Delta(G_F)\le2$.  Suppose $C$ is
an odd-cycle component.  Its orientation has a vertex $v$ with both
an incoming and an outgoing arc, since otherwise sources and sinks
would alternate around the cycle.  By
\eqref{eq:local}, one of $d_D^+(v),d_D^-(v)$ is two.
Reversing all arcs if necessary, assume $d_D^+(v)=2$.  Write the
incoming and outgoing $F$-arcs at $v$ as $uv$ and $vz$, respectively,
and let $vw$ be the other outgoing arc of $v$ in $D$.  The set
\[
F'=(F\setminus\{uv\})\cup\{vw\}=F_v^+
\]
is another minimum feedback arc set.  Since $D$ is oriented,
$w\ne u$, so the degree of $w$ increases by one.  Equation
\eqref{eq:local} for $F'$ implies that $w$ had degree
at most one in $G_F$, and in particular $w\notin V(C)$.  Deleting
$uv$ turns $C$ into a path, and adding $vw$ joins it to the old
component containing $w$, which is also a path.  Every other
component is unchanged, contradicting the choice of $F$.  Thus
every component of $G_F$ is a path or an even cycle, so $G_F$ is
bipartite.

We note that the special feedback arc set can also be found in polynomial time by
alternating cardinality-decreasing replacements that restore
\eqref{eq:local} with the equal-size switches above.
\end{proof}

\begin{proof}[Proof of Theorem~\ref{thm:clw-even}]
Choose a feedback arc set $F$ as in
Lemma~\ref{lem:app-no-clique}.  Equation~\eqref{eq:local}
gives $\Delta(G_F)\le k$.  Moreover, $G_F$ has no $K_{k+1}$
component and, when $k=2$, no odd-cycle component.  Since $k$ is
even, $G_F$ therefore has none of the exceptional components in
$\mathrm{CLW}_k$.  Thus $\mathrm{CLW}_k$ gives an equitable
$k$-coloring of every component.  Combining these colorings
equitably gives an equitable coloring of $G_F$, hence an equitable
acyclic coloring of $D$.
\end{proof}

For the odd case, we use the following consequence of the stronger
characterization of Kierstead and Kostochka~\cite{KKnearly}.

\begin{lemma}\label{lem:app-component-orders}
Let $k\ge3$ be odd and assume $\mathrm{CLW}_k$.  Let $J$ be a graph
with $\Delta(J)\le k$ and no $K_{k+1}$ component.  If $J$ has no
equitable $k$-coloring, then $J$ has exactly one $K_{k,k}$ component and
the order of every component of $J$ is divisible by $k$.
\end{lemma}

\begin{proof}
Suppose $J$ has no equitable $k$-coloring.  By Brooks' theorem, $J$
is $k$-colorable.  Corollary~1 and Theorem~4 of~\cite{KKnearly}
show that $J$ has a component
$S=K_{k,k}$ such that every proper $k$-coloring of $J-V(S)$ has
classes of equal size.  Permuting color labels on any one component
of $J-V(S)$ shows that its classes must also have equal size.
Thus the order of every component of $J$ is divisible by $k$, and
$S$ is the only $K_{k,k}$ component, since $K_{k,k}$ has a proper
$k$-coloring with class sizes $1,3,2,\ldots,2$.
\end{proof}

\begin{proof}[Proof of Theorem~\ref{thm:clw-odd}]
Reversing all arcs preserves acyclic colorings, so it suffices to
consider $\Delta^+(D)\le k$.  Suppose that $D$ has no equitable
acyclic $k$-coloring.  By Lemma~\ref{lem:app-no-clique}, we may
choose a minimum feedback arc set $F$ such that $G=G_F$ has no
$K_{k+1}$ component.  The graph $G$ has no equitable
$k$-coloring, since such a coloring would also color $D$ acyclically.
By Lemma~\ref{lem:app-component-orders}, $G$ has exactly one
$K_{k,k}$ component, say $S$.

Suppose a vertex $v$ of $S$ has both an incoming and an outgoing
$F$-arc.  Since $v\in V(S)$, \eqref{eq:local} and the outdegree bound
give
\[
  k=d_G(v)\le d_D^{\min}(v)\le d_D^+(v)\le k.
\]
Thus equality holds throughout, and we can replace the incoming
$F$-arcs at $v$ by its outgoing arcs not already in $F$, obtaining
the minimum feedback arc set
\[
F'=\bigl(F\setminus\{uv:uv\in F\}\bigr)
   \cup\{vw:vw\in E(D)\setminus F\}.
\]
If $w$ is a new neighbor of $v$, then orientedness implies that the
newly inserted arc $vw$ was not previously represented in $G$ by the
opposite arc $wv$.  Hence
\[
  d_{G_{F'}}(w)=d_G(w)+1
  \le d_D^{\min}(w)\le d_D^+(w)\le k,
\]
where the first inequality is \eqref{eq:local} for the minimum
feedback arc set $F'$.  Thus $d_G(w)\le k-1$, and in particular every
new neighbor lies outside $S$.  The graph $S-v$ remains connected,
and at least one retained outgoing $F$-arc joins $v$ to it.  Thus
$S$ and the old components containing the new neighbors form one
component of $G_{F'}$.  Every former incoming $F$-neighbor of $v$
now has degree $k-1$, so this component is neither $K_{k+1}$ nor
$K_{k,k}$.  Every other component is unchanged.  Since $S$ was the
unique $K_{k,k}$ component of $G$, Lemma~\ref{lem:app-component-orders}
now gives an equitable $k$-coloring of $G_{F'}$, a contradiction.
Thus every vertex of $S$ is a source or a sink
in $F$.  Since $S$ is connected and bipartite, its sides can be
named $A,B$ so that all its arcs in $F$ go from $A$ to $B$.

Choose $b\in B$.  Its feedback degree is $k$, so
\eqref{eq:local} and the outdegree bound give
$d_D^+(b)=k$.  Delete the $k$ arcs from $A$ to $b$ and insert the
$k$ outgoing arcs of $b$, obtaining another minimum feedback arc set
\[
F'=\bigl(F\setminus\{ab:a\in A\}\bigr)
   \cup\{bw:w\in N_D^+(b)\}.
\]
Every new neighbor lies
outside $S$: a new neighbor in $A$ would give a digon in $D$, and a
new neighbor in $B\setminus\{b\}$ would have feedback degree $k+1$,
contrary to \eqref{eq:local} for $F'$.

In $G_{F'}$, the vertices of $S-b$ therefore form a
$K_{k,k-1}$ component.  Let $C$ be the component containing $b$.
All other components are unchanged components of $G$.  If
$C\ne K_{k+1}$, then the new feedback graph has no $K_{k+1}$ component,
and Lemma~\ref{lem:app-component-orders} would force
$|S-b|=2k-1$ to be divisible by $k$, a contradiction.

Thus $C=K_{k+1}$.  Apply the construction in the proof of
Lemma~\ref{lem:app-no-clique} to $C$: at an internal
vertex of its transitive orientation, delete the incoming feedback
arcs and insert the outgoing arcs not already in the feedback set,
obtaining another minimum feedback arc set.
The vertices of $C$ remain connected, and a vertex whose incident
arc was deleted now has degree $k-1$.  Hence the component containing
$C$ is neither $K_{k+1}$ nor $K_{k,k}$.  The construction only
attaches other old components to $C$ and leaves all remaining
components unchanged.  There is therefore no $K_{k+1}$ or $K_{k,k}$
component, so Lemma~\ref{lem:app-component-orders} gives an equitable
$k$-coloring of the resulting feedback graph, again a contradiction.
\end{proof}

\section{Concluding Remarks}\label{sec:conclude}

In this paper, we showed that the Hajnal--Szemer\'edi theorem and its Ore-type strengthening by Kierstead and Kostochka extend to digraphs in the strongest possible sense: the digraph versions stated in
Theorems~\ref{HS} and~\ref{HSO} are not merely analogues of the corresponding graph theorems, they are equivalent to them.  

Combined with the polynomial-time
equitable coloring algorithm of Kierstead, Kostochka, Mydlarz and
Szemer\'edi~\cite{KKMS}, our results yield a polynomial-time algorithm
for finding an equitable acyclic $k$-coloring of every digraph with
$\Dmin(D)\le k-1$.  For the Ore-type versions, the only remaining
algorithmic issue lies in the graph case: it is not known whether the
equitable coloring guaranteed by the Kierstead--Kostochka Ore-type
theorem can be constructed in polynomial time (see
\cite[Conjecture~1]{KKX}).

\begin{conjecture}
For every positive integer $k$ and every $n$-vertex graph $G$ with
$d_G(x)+d_G(y)\le2k-1$ for all $xy\in E(G)$, an equitable $k$-coloring
of $G$ can be found in polynomial (in $n$) time.
\end{conjecture}

Balogh, Kostochka and Treglown~\cite{Ba} also proposed P\'osa- and
Chv\'atal-type degree-sequence strengthenings of the Hajnal--Szemer\'edi
theorem.  The following is an equitable-coloring formulation of their
Chv\'atal-type conjecture\footnote{To see the equivalence when $k\nmid n$, set $t=kq-n$ and add a
disjoint copy of $K_t$ (nothing if $t=0$).  The resulting graph has
$kq$ vertices and satisfies the divisible form of the same
degree-sequence condition.  Indeed, if
$d_{(q-1)i}\le i-1$, then the original graph already has at least
$(q-1)i$ vertices of degree at most $i-1$.  If instead
$d_{n-i+1}\le2k-i-1$, then at most $i-1$ original vertices have degree
greater than $2k-i-1$, and no new vertex does, since each has degree
$t-1\le k-1\le2k-i-1$.  In an equitable $k$-coloring, the vertices of
$K_t$ receive distinct colors; deleting them leaves color classes of
sizes $q-1$ and $q$.}.  Treglown~\cite{Treglown2016} proved an asymptotic version of their
P\'osa-type conjecture.

\begin{conjecture}[Balogh, Kostochka and Treglown~\cite{Ba}]
\label{conjBKTcolor}
Let $n,k\in\mathbb{N}$ with $1\le k<n$, and set $q=\ceil{\frac{n}{k}}$.
If $G$ is an $n$-vertex graph with degree sequence
$d_1\le\cdots\le d_n$ such that, for all $1\le i\le k$,
\begin{equation}\label{eq:chvatal}
  d_{(q-1)i}\le i-1
  \qquad\text{or}\qquad
  d_{n-i+1}\le 2k-i-1,
\end{equation}
then $G$ has an equitable $k$-coloring.
\end{conjecture}

The same reduction from Section~\ref{sec:reduction} gives the following
equivalent digraph formulation.

\begin{conjecture}\label{conjBKTdigraphcolor}
Let $n,k\in\mathbb{N}$ with $1\le k<n$, and set $q=\ceil{\frac{n}{k}}$.
Let $D$ be an $n$-vertex digraph, and write
$a_1\le\cdots\le a_n$ for the nondecreasing sequence of the
mindegrees $\dmin(v)$ of its vertices.  If, for all $1\le i\le k$,
\[
  a_{(q-1)i}\le i-1
  \qquad\text{or}\qquad
  a_{n-i+1}\le 2k-i-1,
\]
then $D$ has an equitable acyclic $k$-coloring.
\end{conjecture}

When $k=\frac{n}{2}$, it is well-known that Ore's
degree condition implies P\'osa's condition, which in turn implies
Chv\'atal's condition.  However, this is not true when $k\leq \frac{n}{3}$.  Taking complements in the construction of
Lang and Sanhueza-Matamala~\cite[Proposition~A.4]{LangSanhuezaMatamala}
shows that the Ore-type coloring condition does not imply the
P\'osa-type condition.  A modification of their example shows that it
does not imply the Chv\'atal-type condition either.\footnote{A small
example with $n=10$ and $k=2$ is $3P_3\cup K_1$.  It satisfies
$d(x)+d(y)\le3$ for every edge $xy$, but fails the condition in
Conjecture~\ref{conjBKTcolor} at $i=2$.}

\begin{proposition}\label{prop:ore-not-chvatal-color}
Let $r\ge3$ and $t\ge2$ be integers, and set $k=5t+1$, $n=rk$, and $m=\left\lfloor\frac{t-2}{r-1}\right\rfloor$.
There exists an $n$-vertex graph $H$ such that
$d_H(x)+d_H(y)=2k-2$
for every edge $xy$, but whose degree sequence
$d_1\le\cdots\le d_n$ violates \eqref{eq:chvatal} at every index $k-m\le i\le k$ (each by at least $\frac{t}{2}$).
\end{proposition}

\begin{proof}
Construct a graph $H_0$ from six disjoint independent sets, arranged
in cyclic order, with sizes $2t, 3t, 2t, 3t,2t,3t$. Add all edges between 
consecutive sets, including the first and last, and no other edges.  Thus
$H_0$ is a blow-up of $C_6$.

Let $H$ be the disjoint union of $H_0$ and
$3+(r-3)k$ isolated vertices, so
\[
  |H|=15t+3+(r-3)k=rk=n.
\]
Every edge of $H$ joins a vertex of degree $6t$ to a vertex of degree
$4t$, so its endpoints have degree sum $10t=2k-2$.

The degree sequence of $H$ consists of $3+(r-3)k$ zeros, followed by
$9t$ entries equal to $4t$ and $6t$ entries equal to $6t$.  Let
$0\le j\le m$ and set $i=k-j$.  Since
\[
  (r-1)i=n-k-(r-1)j\ge n-k-(t-2)=n-6t+1
\]
and $n-i+1\ge n-k+1=n-5t$, both indexed degrees in \eqref{eq:chvatal} are equal to $6t$.  So their excesses over the
two upper bounds are
\[
  d_{(r-1)i}-(i-1)
    =t+j ~~\text{ and }~~
  d_{n-i+1}-(2k-i-1)
    =t-j.
\]
Since $r\ge3$, we have $j\le m\le\frac{t-2}{2}$, so both excesses are at
least $\frac{t}{2}$.
\end{proof}

\vspace{0.5cm}
\noindent \textbf{\large AI declaration.}
The reduction given in this paper arose
during a discussion between the second author and ChatGPT 5.6 Sol attempting to locate the bottleneck in extending the results of \cite{Czygrinow2015} to prove Conjecture \ref{con:treglown}.  Instead of locating the bottleneck, the chatbot gave a clever proof which shows that Conjecture \ref{con:treglown} reduces to the original Hajnal-Szemer\'edi theorem. 
After further discussion about the originality of this idea, the chatbot identified earlier work of Aboulker, Oijid, Petit, Rocton, and Simon~\cite{AOPRS} (building on earlier work regarding median orders of tournaments~\cite{HT} and feedback arc sets~\cite{Younger}) that uses the undirected graph associated with a feedback arc set to transfer results about independent sets in graphs to acyclic sets in digraphs while preserving the relevant degree conditions. So the main result here of showing that Hajnal-Szemer\'edi's theorem implies Treglown's conjecture already follows from \cite{AOPRS}.  The relevant argument in \cite{AOPRS} is formulated using an ordering which minimizes the number of backward arcs.  Finding such an ordering is NP-hard, since it is equivalent to the minimum feedback arc set problem~\cite{Karp}.  However, only local optimality under moving a vertex to the beginning or end is needed.  Thus we make it explicit that the corresponding construction can be carried out in polynomial time.
Separately, the first, fourth, and fifth authors initially used AI to prove Conjecture \ref{con:treglown} and Problem \ref{pro1}. After subsequent discussions, the authors decided to combine the two independent papers into one, yielding a stronger result.  After combining papers, we added a new section on oriented graphs which raises some interesting new questions which are not immediately settled by the reduction to graphs.

\vspace{0.5cm}
 \noindent \textbf{\large Acknowledgment.}  We thank Lucas Picasarri-Arrieta for helpful comments on the first draft of this paper.

\end{document}